\documentclass[12pt]{amsart}

\usepackage{amsmath,amsthm,color,eepic}
\usepackage{graphicx,xypic}
\usepackage{amssymb} \usepackage{verbatim}
\usepackage{hyperref}

\xyoption{poly}
\xyoption{curve}

\usepackage[foot]{amsaddr}
\usepackage{geometry}
\theoremstyle{theorem}
\newtheorem{theorem}{Theorem}
\newtheorem{proposition}{Proposition}
\newtheorem{lemma}{Lemma}

\theoremstyle{definition}
\newtheorem{definition}{Definition}

\newtheorem{remark}{Remark}

\newtheorem{example}{Example}

\def\rh{\hat{\rho}}
\def\drh{\Delta_{\hat{\rho}}}
\def\dhat{\hat{d}}
\def\drho{{\Delta_\rho}}

\def\srp{$s,r$-path}
\def\srt{$s,r$-trail}
\def\srw{$s,r$-walk}
\def\Prob{{\text{Prob}}}
\def\pprime{{\prime\prime}}
\def\ppprime{{\prime\prime\prime}}

\begin{document}

\title{\bf Expected reliability of communication protocols}

\date{}

\author{Andr\'e K\"undgen}\address{Department of Mathematics, California State University San Marcos, San Marcos, CA 92096}\email{akundgen@csusm.edu}\thanks{The first author was supported by ERC Advanced Grant GRACOL, project no. 320812 while visiting Denmark Technical University.}

\author{Janina Patno} 
\keywords{Network, communication protocol, two-terminal graph,  two-terminal reliability polynomial}

\maketitle

\begin{abstract}
We consider the problem of sending a message from a sender $s$ to a receiver $r$ through an unreliable network by specifying in a protocol what each vertex is supposed to do if it receives the message from one of its neighbors. A protocol for routing a message in such a graph is {\em finite} if it never floods $r$ with an infinite number of copies of the
message. The {\em expected reliability} of a given protocol is the probability that a message sent from $s$ reaches $r$ when
the edges of the network fail independently with probability $1-p$.

We discuss, for given networks, the properties of finite protocols with maximum expected reliability in the case when
$p$ is close to 0 or 1, and we describe networks for which no one protocol is optimal for all values of $p$. In general, finding
an optimal protocol for a given network and fixed probability is challenging and many open problems remain.
\end{abstract}

\section{Introduction}\label{sec:intro}

All graphs $G$ in this paper will be undirected, simple (no loops or multiple edges), and have two distinct vertices identified as $s$ (the sender) and $r$ (the receiver). Our goal is to pass a message from $s$ to $r$ along the edges of $G$, where we will assume that edges may fail, but vertices do not. We will also assume that vertices are memoryless and have no information about which edges have failed and which are alive, but rather, will pass the message along based on a set of instructions that are given before we know which edges fail. More formally, an {\em instruction} is any triple $uvw$ where $u,w$ are different neighbors of $v$ and the idea is that if $v$ receives a message from $u$, then it passes it on to $w$. (We require that $u\neq w$, since there is no point in sending a message back where it came from. Also observe that $uvw$ and $wvu$ are different instructions, whereas for edges we have $uv=vu$ as usual.) A {\em protocol} $A$ is a set of instructions. For a protocol $A$ of a graph $G$ we now try to send a message from $s$ to $r$ by sending one copy of the message from $s$ to everyone of its neighbors, and whenever an intermediate vertex receives the message, it passes it on as described by $A$. 

There are several considerations of what makes a good protocol. A very basic one is that we do not want to flood the receiver, that is, we do not want $r$ to receive infinitely many copies of the message. This is what would usually happen if no edge fails and every intermediate vertex simply sends the message to every neighbor each time it receives it. This basic consideration is studied in~\cite{FKPR}. Another consideration is that we would like our protocol to have a certain robustness, that is the failure of only few edges does not interrupt communication. This concept was suggested to the first author by Lov\'asz and is studied in~\cite{KPR}. A second approach suggested by Lov\'asz~\cite{KPR} is to study the probability that a message sent from $s$ will reach $r$ when the edges of $G$ only survive with some probability $p$. This approach was first studied in the Master's thesis of the second author~\cite{Pat}. In this paper we build on the ideas from~\cite{FKPR,KPR,Pat} to investigate the probabilistic setting.
\section{The basic model}\label{sec:basic}

To study this communication model we need a few basic definitions that are consistent with those used in~\cite{FKPR,KPR,Pat,Wes}. A \emph{$u,v$-walk of length $k$} in a graph $G$ is a sequence $v_0, v_1, v_2, ..., v_{k-1}, v_k$ of vertices such that $v_0=u$, $v_k=v$ and  $v_{i-1},v_i$ are adjacent for all $i$ with $1 \leq i \leq k$. We say that an instruction $uvw$ is {\em contained} in this walk if there is an index $i$ such that $u=v_{i-1}, v=v_i, w=v_{i+1}$. 
A \emph{trail} is a walk such that if $v_i = v_j$ for $i \neq j$, then $v_{i+1} \neq v_{j+1}$. Thus, each edge $uv$ can be used at most twice in a trail: once as $u,v$ and once as $v,u$. (This is a slightly nonstandard use of this term.)
A walk (trail) is $closed$ if the endpoints, $v_0$ and $v_{k}$, of the walk (trail) are the same. 
A \emph{path} is a walk \mbox{$v_0, v_1, v_2, ..., v_{k-1}, v_k$} such that $v_0, v_1, ..., v_{k-1}, v_k$ are distinct vertices of $G$. 
For a protocol $A$, an {\em $A$-walk} is an \srw~such that every instruction it contains is in $A$. The concepts of {\em $A$-trail} and {\em $A$-path} are defined similarly. 

Using this notation it is now easy to see that $r$ receives the message exactly once for every $A$-walk whose edges do not fail. To have a good protocol we generally want there to be many $A$-walks, but we also do not want to flood the receiver with infinitely many messages when no edge fails. Thus we call a protocol $A$ {\em finite} if there are only finitely many $A$-walks. To characterize such protocols we call an instruction {\em (strongly) essential} for $A$ if it is contained in an $A$-walk (an $A$-path). If every instruction of $A$ is (strongly) essential, then we call $A$ a {\em (strongly) essential protocol}. One of the key lemmas of~\cite{FKPR} is that a protocol $A$ is finite if and only if it does not contain an {\em essential circuit}, that is a closed trail such that every instruction it contains (including $v_{k-1}v_0v_1$) is essential for $A$. 

The simplest protocol is the {\em Complete Forwarding Protocol} (CFP) $A^*=\{ uvw: uvw$ is contained in some \srp$\}$. By definition $A^*$ is strongly essential. Observe that there can be instructions that are contained in an \srw, but not in an \srp~and such instructions would not be in $A^*$.

\begin{figure}[ht]
\centerline{
\hbox{\xy /r3.8pc/:,
(0,0)*+{s}*\cir{}="B0",
(2,0),
{\xypolygon4"A"{~>{}}},
(4,0)*+{r}*\cir{}="B1",
"A0"*{\bullet}, "A1"*{\bullet},
"A2"*{\bullet}, "A3"*{\bullet},"A4"*{\bullet},
"A0"+<10pt,0pt> *{3},
"A1"+<2pt,-10pt> *{4},
"A2"+<-2pt,-10pt> *{1},
"A3"+<-2pt, 10pt> *{2},
"A4"+<2pt,10pt> *{5},
"A0";"A1"**@{-},
"A0";"A2"**@{-},
"A0";"A3"**@{-},
"A0";"A4"**@{-},
"A1";"A2"**@{-},
"A3";"A4"**@{-},
"B0";"A2"**@{-},
"B0";"A3"**@{-},
"B1";"A1"**@{-},
"B1";"A4"**@{-},
\endxy}}
\caption{$B_0$}\label{fig:B0}
\end{figure}

\begin{example}\label{B0}
The graph $B_0$ in Figure~\ref{fig:B0} has Complete Forwarding Protocol $A^*=\{s13, s14,$ $s23,$ $s25, 132, 134, 135, 143, 14r, 231, 234, 235, 253, 25r, 314, 325, 34r, 35r, 432, 435,   531, 534\}.$

Observe that $413\notin A^*$, even though there is an \srw~$s,2,3,4,1,3,5,r$. Every essential circuit must contain the edges of a cycle, but not $s$ or $r$. It is easy to see that $C=1,4,3,2,5,3,1$ is the only essential circuit of $A^*$ up to the choice of the starting point.
\end{example}

The graph $B_0$ shows that the CFP need not be finite. In fact the main result of~\cite{FKPR} is a characterization of the graphs for which the CFP is finite in terms of 10 forbidden minors, one of which is $B_0$. Any protocol $A$ with $A\subseteq A^*$ is called a {\em Partial Forwarding Protocol} (PFP).  An SPFP is a strongly essential PFP, that is a protocol $A$ in which every instruction is contained in an $A$-path. Our first lemma will imply that in general it suffices to study SPFP's.

\begin{lemma}\label{PFP lemma}
If $A$ is a protocol for a graph $G$, then there is a PFP $A'$ with the following properties:
\begin{enumerate}
\item The edge-set of every $A$-walk contains an $A'$-path.
\item The edge-set of every $A'$-walk contains an $A'$-path.
\item Every instruction in $A'$ is contained in an $A'$-path (and thus strongly essential). 
\item If $A$ is finite, then $A'$ is finite.
\end{enumerate}
\end{lemma}

\begin{proof}
Observe that every $A$-walk $W=v_0,v_1,\dots,v_k$ contains an \srp~$P=u_0,\dots, u_m$ such that for every $u_i$ we can find indices $j\le j'$ with $v_j=u_i=v_{j'}$, $u_{i-1}=v_{j-1}$ and $u_{i+1}=v_{j'+1}$. Moreover, there are only finitely many \srp s in $G$, so we can let the paths obtained in this way be $P_1,\dots,P_k$, and $W_i$ be some $A$-walk that contains $P_i$ in this way. Let
$A'=\{ uvw: uvw$ is contained in some $P_i\}$. $A'$ is a PFP, and since every $P_i$ is an $A'$-path, (a) and (c) clearly hold.

To prove (d), suppose $A'$ is not finite and there is an essential circuit $C'=(u_0,u_1,\dots, u_k)$ for $A'$. Every instruction $u_{i-1}u_iu_{i+1}$ in this circuit must be contained in some $P_i$. Now by construction the $A$-walk $W_i$ contains a subwalk $v_{j-1},v_j,\dots,v_{j'},v_{j'+1}$ that includes only instructions that are essential for $A$, and such that $u_i=v_j=v_{j'}$, $u_{i-1}=v_{j-1}$ and $u_{i+1}=v_{j'+1}$. Thus if in $C'$ we replace every vertex $u_i$ by the corresponding walk $v_j,\dots, v_{j'}$, then we obtain an essential circuit $C$ for $A$. Thus $A$ is not finite.

To prove (b) consider $A^\pprime$ obtained from $A'$ by repeating the same procedure, that is $A^\pprime=(A')'$. By (c) for $A'$ it follows that $A'\subseteq A^\pprime$, so that every $A'$-path is an $A^\pprime$-path. Thus (a) holds for $A^\pprime$ instead of $A'$, and (c) and (d) follow similarly with $A^\pprime$ instead of $A'$. If $A'$ does not satisfy (b), then we get that one of the paths $P_i$ in the construction of $A^\pprime$ from $A'$ is not an $A'$-path, so that $A^\pprime$ has an instruction that is not in $A'$, that is $A'\subset A^\pprime$. Repeating this procedure we get a strictly increasing sequence $A'\subset A^\pprime\subset A^\ppprime\subset\dots $ of protocols that satisfy (a,c,d). Since there are only finitely many instructions, this process terminates in a protocol satisfying all four conditions, which will be our $A'$.   
\end{proof}

\section{The probabilistic model}\label{sec:prob}

Suppose every edge of $G$ fails with probability $1-p$ and survives with probability $p$, where $p$ is usually fixed in (0,1), but for the purpose of this section $p$ need not be constant on $E(G)$. We define the {\em (expected) reliability} of a protocol $A$ for $(G,p)$, denoted by $\rho_A(G,p)$ or simply $\rho_A$ if $G$ and $p$ are clear from the context, to be the probability that a message sent from $s$ under protocol $A$ reaches $r$. Note that this is defined, whether $A$ is finite or not. More formally

\begin{definition}
Let $A$ be any protocol for $G$, and $p:E(G)\to(0,1)$. Then $\rho_A(G,p)$ and $\rho'_A(G,p)$ are the probability that the edges of some $A$-walk (respectively $A$-path) do not fail if every edge $e\in E(G)$ fails independently with probability $1-p(e)$. Moreover, $\rh(G,p)=\max\{\rho_A(G,p): A$ is finite$\}$.
\end{definition}

Note that $\rho'_A(G,p)\le\rho_{A}(G,p)$ and the former is usually easier to determine, but equality need not hold. Observe also that $\rh$ is well-defined since there are only finitely many protocols $A$, but that for different choices of $p$ this maximum might be achieved for different protocols $A$. Also note that the well studied {\em (two terminal)-reliability} $\rho (G,p)$ (see~\cite{Col,GM,Kel,MS,PB,YG}) is the probability that $s,r$ are in the same component of $(G,p)$. This is identical to the probability that some $s,r$-path survives in $(G,p)$, that is $\rho(G,p)=\rho'_{A^*}(G,p)=\rho_{A^*}(G,p)$. Lemma~\ref{PFP} immediately implies the following.

\begin{proposition}\label{PFP}
For every protocol $A$ there is a SPFP $A'$ such that $\rho_A(G,p)\le\rho_{A'}(G,p)=\rho_{A'}'(G,p)$ for every $p:E(G)\to(0,1)$. 
$\rho_A(G,p)=\rho_{A'}'(G,p)$ if and only if every $A'$-path is an $A$-path.
Moreover if $A$ is finite, then $A'$ is finite.
\end{proposition}

\begin{proof} For given $A$ let $A'$ be as in Lemma~\ref{PFP lemma}. By (c) $A'$ is an SPFP and by (d) $A'$ is finite when $A$ is finite. Now if we let $Y_i$ be the event that the edges of the $A'$-path $P_i$ survive, $X$ be the event that an $A$-walk from $s$ to $r$ survives, and $Z$ be the event that an $A'$-walk from $s$ to $r$ survives, then $X\subseteq Y_1\cup Y_2\cup\dots Y_k\subseteq Z$ by (a) and the fact that every $A'$-path is an $A'$-walk. Moreover the second containment is equality by (b). Thus $\rho_A(G,p)=\Prob(X)\le\Prob(Y_1\cup\dots\cup Y_k)=\rho_{A'}'(G,p)=\Prob(Z)=\rho_{A'}(G,p)$, as desired. 

If every $A'$-path is an $A$-path, then $\rho'_{A'}\le \rho'_A$. Since every $A$-path is an $A$-walk, it now follows that $\rho'_A\le\rho_A\le\rho'_{A'}\le\rho'_A$ and the desired equality holds. Now suppose that some $A'$-path $P_i$ is not an $A$-path. Let $Y$ be the event that only the edges in $P_i$ survive, but all other edges fail. Clearly $Y\subseteq Y_i$ and $\Prob(Y)>0$ since every edge $e$ has $0<p(e)<1$. Moreover, $E(P_i)$ does not contain an $A$-walk, since otherwise we would get the contradiction that either $P_i$ is an $A$-path, or that the $A$-walk contains an instruction of the form $uvu$. Thus $X\subseteq(Y_1\cup\dots\cup Y_k)-Y$, and we get that $\rho_A<\rho_{A'}'$.
\end{proof}

\begin{definition} 
We call a finite SPFP $A$ {\it optimal} for $(G,p)$ if for every finite protocol $A'$ we have $\rho_{A'}(G,p)\le \rho_A(G,p)$, and $\rho_A(G,p)$ is the probability that some $A$-path survives.
\end{definition}

Proposition~\ref{PFP} immediately implies that for every $(G,p)$ there is an optimal SPFP $A$ and 
$$\rh(G,p)=\max\{\rho'_A(G,p): A {\text{~is~a~finite~SPFP}}\}.$$ 

For an (optimal) SPFP $A$  we can compute $\rho_A(G,p)=\Prob(Y_1\cup\dots \cup Y_k)$ by Inclusion-Exclusion once we know all $A$-paths $P_i$. Thus $\rho_A(G,p)$ is a polynomial in $p$ if $p$ is constant. Moreover, since there are only finitely PFP's it follows that $\rh(G,p)$ is piecewise polynomial in $p$ if every edge has the same probability $p$.

\section{Series-parallel replacements}\label{sec:series}

In this section we present a method for building large graphs whose reliability can be computed easily.

Given graphs $G_1$, $G_2$ with senders $s_{1},s_2$ and receivers $r_{1},r_{2}$ respectively, we can obtain a graph $H$ with sender $s$ and receiver $r$ by \emph{series operation}, written $H = G_1 \circ G_2$,  by setting $s = s_{1}$, $ r = r_{2}$, and identifying $r_1, s_{2}$ with a new vertex $x$. We obtain $H$ by \emph{parallel operation}, written $G_1 || G_2$, by identifying  $s=s_{1}=s_{2}$ and $r= r_{1} = r_{2}$. Any graph that can be built from $K_2$ using only these operations is called {\em series-parallel}. 

\begin{proposition}\label{series circuit}
Let $H$ be a series-parallel graph, $A_{s,r}$ the CFP in $H$ and $A_{r,s}$ be the CFP in $H$ if we interchange $s$ and $r$.
If $A_H\subseteq A_{s,r}\cup A_{r,s}$, then every walk $W=v_0,v_1,\dots,v_k$ such that $v_{i-1}v_iv_{i+1}\in A_H$ for all $i$ with $1\le i\le k-1$ is a path.
Specifically if $W$ is nontrivial, then it is not closed.
\end{proposition}

\begin{proof} The first statement immediately implies the second, and so it suffices to prove the former.

Using the recursive definition of $H$ it is easy to show that there is an injection $f:V(H)\to[0,1]$ such that $f$ is increasing along every \srp. 
Thus $f$ is strictly increasing along every instruction in $A_{s,r}$, and strictly decreasing along every instruction in $A_{r,s}$. Suppose $W$ is such a walk. Since $H$ has no loops we may assume that $k\ge 2$ and $f(v_0)\neq f(v_1)$. If $f(v_0)<f(v_1)$, then it follows that $v_0v_1v_2\in A_{s,r}$ and thus $f(v_1)<f(v_2)$. Continuing along it follows that $f$ is strictly increasing along $W$ and thus all $v_i$ must be distinct. If $f(v_0)>f(v_1)$, then it follows similarly that $f$ is strictly decreasing along $W$.
\end{proof}

Proposition~\ref{series circuit} implies that the CFP $A^*$ for a series-parallel graph $H$ is finite (as it has no essential circuit), and thus  $\rh(H)=\rho(H)$.
Moreover, it is easy to compute $\rho(H)$ when $H$ is series parallel, since in general $\rho(G_1\circ G_2)=\rho(G_1)\rho(G_2)$ and $\rho(G_1||G_2)=\rho(G_1)+\rho(G_2)-\rho(G_1)\rho(G_2)$. Our next result generalizes these equations and exploits the fact that $p:E(G)\to(0,1)$ need not be constant.

\begin{definition}
Let $G_1$, $H$ be graphs with specified vertices $s,r$. If $e_1=xy$ is an edge in $G_1$, then let the $G$ be the graph obtained from $G_1-e_1$ by identifying $s,r$ in $H$ with $x,y$ in $G_1$. (See Figure~\ref{fig:B0 replace s1,s2} for an example.) We call $G$ the {\em expansion} of $G_1$ at $e_1$ by $H$. If $p$ is a probability distribution on $E(G)$, then the {\em implied distribution} $p_1$ on $G_1$ is given by $p_1(e)=p(e)$ for $e\neq e_1$, and $p_1(e_1)=\rho(H,p|_H)$, where $p|_H$ is the restriction of $p$ to $E(H)$. 

If $uvw$ is an instruction in $G_1$, then the {\em corresponding} set of instructions $I(uvw)$ in $G$ is given by $I(uvw)=\{uvw\}$ (when $e_1\neq uv,vw$), $I(uvw)=\{u'vw: u'\in N_H(v)\}$ (when $e_1=uv$) and  $I(uvw)=\{uvw': w'\in N_H(v)\}$ (when $e_1=vw$,) where $N_H(v)$ denotes the set of neighbors of $v$ in $H$. If $A$ is a set of instructions in $G_1$, then we let $I(A)=\bigcup_{uvw\in A}I(uvw)$.     

Let $A_{u,v}$ denote the CFP on $H$ with $s=u,r=v$ and let $A$ be any protocol for $G_1$. If there are some $A$-paths in $G_1$ containing $xy$ in which $x$ directly precedes $y$ and some in which $y$ directly precedes $x$, then we define $A_H=A_{x,y}\cup A_{y,x}$. If no $A$-path in $G_1$ contains $xy$ we let $A_H=\emptyset$. If in every $A$-path in $G_1$ that contains $xy$ we have that $x$ immediately precedes $y$ ($y$ immediately precedes $x$), then we let $A_H=A_{x,y}$ and $A_H=A_{y,x}$ respectively. In any case $A_H$ satisfies the hypothesis of Proposition~\ref{series circuit} when $H$ is series-parallel. With this notation we will define the {\em extension} of $A$ to $G$ by $A^+=I(A)\cup A_H$.
\end{definition}

\bigskip
\begin{figure}[ht]
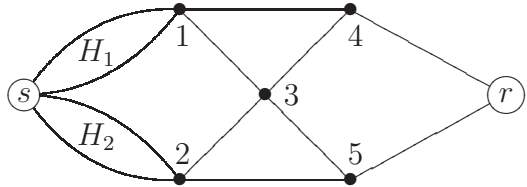

\centerline{
\hbox{\xy /r3.8pc/:,
(0,0)*+{s}*\cir{}="B0",
(2,0),
{\xypolygon4"A"{~>{}}},
(4,0)*+{r}*\cir{}="B1",
"A0"*{\bullet}, "A1"*{\bullet},
"A2"*{\bullet}, "A3"*{\bullet},"A4"*{\bullet},
"A0"+<10pt,0pt> *{3},
"A1"+<2pt,-10pt> *{4},
"A2"+<1pt,-10pt> *{1},
"A3"+<1pt, 10pt> *{2},
"A4"+<2pt,10pt> *{5},
"A0";"A1"**@{-},
"A0";"A2"**@{-},
"A0";"A3"**@{-},
"A0";"A4"**@{-},
"A1";"A2"**@{-},
"A3";"A4"**@{-},
"B0";"A2"**\crv{(0.5,0.75)},
"B0";"A2"**\crv{(0.75,0)},
"B0";"A3"**\crv{(0.5,-0.75)},
"B0";"A3"**\crv{(0.75,0)},
"B1";"A1"**@{-},
"B1";"A4"**@{-},
(0.6,0.35)*+{H_1},
(0.6,-0.35)*+{H_2},
\endxy}}
\caption{$G$ obtained from $B_0$ by replacing $s1$, $s2$ by $H_1, H_2$}\label{fig:B0 replace s1,s2}
\end{figure}

\begin{proposition}\label{prop:replace}
Let $H$ be a series-parallel graph, and $G$ be the expansion of some graph $G_1$ at some edge $e_1=xy$ by $H$.
If $(G,p)$ is a given probability distribution and $p_1$ is its implied distribution on $G_1$, then the following hold:
\begin{enumerate}
\item For every protocol $A$ for $G_1$: $\rho'_A(G_1,p_1)=\rho'_{A^+}(G,p)$.
\item $\rho(G,p)=\rho(G_1,p_1)$.
\item If $A, A'$ are protocols for $G_1$, then $A^+\subseteq (A')^+$ iff $A\subseteq A'$.
\item For every protocol $A$ for $G_1$: $A$ is a finite SPFP for $G_1$ if and only if $A^+$ is a finite SPFP for $G$.
\item $\rh(G,p)\ge \rh(G_1,p_1)$ with equality if $e_1$ is in no essential circuit for the CFP of $G_1$. 
\end{enumerate}
\end{proposition}

\begin{proof}
Let $T_{u,v}$ be the set of $u,v$-paths in $H$. For every path $P=v_0,\dots,v_k$ in $G_1$ let $P^+$ be the following family of paths in $G$: If $e_1=xy$ is not on $P$, then $P^+=\{P\}$. If $v_i=x, v_{i+1}=y$, then $P^+=\{v_0,v_1,\dots,v_{i-1}, T,v_{i+2},\dots,v_k: T\in T_{x,y}\}$ and if $v_i=y, v_{i+1}=x$, then $P^+=\{v_0,v_1,\dots,v_{i-1}, T,v_{i+2},\dots,v_k: T\in T_{y,x}\}$. 

Let $A$ be any protocol for $G_1$. If $P$ is an $A$-path, then it follows that $P^+$ is a family of $A^+$-paths. Similarly if some member of $P^+$ is an $A^+$-path, then $P$ is an $A$-path. So the sets $P^+$ form a partition of the family of $A^+$-paths. If $X_i$ is the event that all edges of the $A$-path $P_i$ survive in $(G_1,p_1)$ and $Y_i$ is the event that the edges of some path in $P_i^+$ survive in $(G,p)$, then it is not hard to see that for every collection of indices $J$, $\Prob_{(G_1,p_1)}(\bigcap_{i\in J}X_i)=\Prob_{(G,p)}(\bigcap_{i\in J}Y_i)$. Thus by inclusion-exclusion it follows that (a) holds:
$$\rho'_A(G_1,p_1)=\Prob_{(G_1,p_1)}(X_1\cup X_2\cup\dots\cup X_k)=\Prob_{(G,p)}(Y_1\cup Y_2\cup\dots\cup Y_k)=\rho'_{A^+}(G,p).$$

Let $A^*_1$ be the CFP on $G_1$ and $A^*$ be the CFP on $G$. Since every \srp~in $G$ is in $P^+$ for some \srp~$P$ in $G_1$, it follows that $A^*=(A^*_1)^+$.
Thus $\rho(G_1,p_1)=\rho'_{A^*_1}(G_1,p_1)=\rho'_{(A^*_1)^+}(G,p)=\rho(G_1,p_1)$, and (b) holds.

If $A\subseteq A'$, then $I(A)\subseteq I(A')$ and $A_H\subseteq A'_H$, so that $A^+\subseteq (A')^+$.
If $A^+\subseteq (A')^+$, then  $I(A)\subseteq I(A')$ and thus $A\subseteq A'$, so that (c) follows. 

$A$ is a PFP for $G_1$ iff $A^+$ is a PFP for $G$ follows by combining (c) and $A^*=(A^*_1)^+$. Let $uvw\in A$. Then $uvw$ is strongly essential for $A$ iff $uvw$ is contained in some $A$-path $P$ iff every $xyz\in I(uvw)$ is contained in some member of $P^+$ for some $A$-path $P$ iff every $xyz\in I(uvw)$ is in an $A^+$-path iff every $xyz\in I(uvw)$ is strongly essential for $A^+$. Furthermore every element of $A_H$ is trivially strongly essential for $A^+$ by definition. Thus $A$ is an SPFP iff $A^+$ is an SPFP. 

It remains to consider finiteness, where we may now assume that every instruction in $A$ and $A^+$ is essential.
Replacing every occurence of the edge $e_1$ in an essential circuit for $A$ in $G_1$ with a path from $T_{x,y}$ or $T_{y,x}$ as appropriate it is easy to see that we obtain an essentail circuit for $A^+$ in $G$. So suppose $G$ contains an essential circuit $C$ for $A^+$. Then Proposition~\ref{series circuit} implies that $C$ cannot be entirely contained in $H$ (and thus only contain instructions from $A_H$), and that if $C$ enters $H$ at one of $x,y$, then it must leave it at the other.  Thus if we remove all vertices in $V(H)-\{x,y\}$ from the sequence $C$, then we get a closed walk $C'$ in $G_1$. Moreover, by construction of $A^+$ it follows that every instruction $uvw$ contained in $C'$ is in $A$, so that $C'$ must be an essential circuit in $G_1$ and (d) is proven.

For (e) it follows so far that
\begin{eqnarray*}
\rh(G_1,p_1)&=&\max\{\rho'_A(G_1,p_1): A {\text{~is~a~finite~SPFP~for~}}G_1\} \\
 &=&\max\{\rho'_{A^+}(G,p): A {\text{~is~a~finite~SPFP~for~}}G_1\} \\
&=&\max\{\rho'_{A^+}(G,p): A^+ {\text{~is~a~finite~SPFP~for~}}G\} \\
&\leq&\max\{\rho'_B(G,p): B {\text{~is~a~finite~SPFP~for~}}G\} =\rh(G,p).
\end{eqnarray*}

For equality it remains to show that for every finite SPFP $B$ for $(G,p)$ there is a finite SPFP $A$ for $(G_1,p_1)$ with $\rho'_B(G,p)\leq \rho'_{A^+}(G,p)$. 
For a given $B$, consider $A=\{uvw\in A_1^*: I(uvw)\cap B\neq\emptyset\}$. We first show that $B\subseteq A^+$, since then $\rho'_B(G,p)\leq \rho'_{A^+}(G,p)$ is trivial.
So let $abc\in B$ be given. Since $B$ is an SPFP there is a $B$-path $Q$ containing $abc$. For this path $Q$ in $G$ there must be a path $P$ in $G_1$ with $Q\in P^+$.
Now for every instruction $uvw$ contained in $P$ we have that $I(uvw)$ has an instruction in $Q$ and thus $B$. Hence $P$ is an $A$-path, and thus $Q\in P^+$ is an $A^+$-path. Specifically $abc\in A^+$, as desired.

It remains to show that $A$ is a finite SPFP. $A$ is a PFP since $A\subseteq A_1^*$. If $uvw\in A$, then there is $abc\in B$ with $abc\in I(uvw)$. 
As in the previous argument, $abc$ is in some $B$-path $Q$, and there is an $A$-path $P$ with $Q\in P^+$. Since $abc$ is in $Q$ and $I(uvw)$ we have that $uvw$ is an instruction in (the $A$-path) $P$. Thus $A$ is an SPFP.
Finally, suppose that $A$ is not finite, that is it contains an essential circuit $C=v_0,\dots,v_k$ in $G_1$. Since by assumption $C$ does not use $e_1$ then every instruction in $A$ is also in $B$, so that $C$ is an essential circuit in $B$, a contradiction. 
\end{proof}

\section{Discrepancies}\label{sec:disc}

In general we are more interested in finding an optimal protocol for $(G,p)$, than the actual value of $\rh(G,p)$. Since $\rh$ can be very close to $\rho$ it makes sense to study the difference between these parameters.

\begin{definition}
Let $A^*$ be the CFP on $G$. For a set of instructions $I\subseteq A^*$, we define its {\em discrepancy} $d_I$ as $d_I=d_I(G,p)=\rho(G,p)-\rho_{A^*-I}(G,p)$.
The {\em minimum discrepancy} $\dhat$ of $G$ is defined as $\dhat=\dhat(G,p)=\rho(G,p)-\rh(G,p)=\min\{d_I(G,p): A^*-I$ is finite$\}$.
\end{definition}

Thus $d_I$ is the probability that some $s,r$-walk remains, but that every such $s,r$-path contains an instruction in $I$. 
Observe that if $I_1\subset I_2$, then it follows directly that $d_{I_1}\le d_{I_2}$, that is $d_I$ is monotone in $I$.
The following result is easy to see, but a proof is Lemma 3.3.5 of~\cite{Pat} with the notation for $d_I$ being $d_{A^*-I}$.

\begin{lemma}\label{def lemma}
Let $(G, p)$ and $I\subseteq A^*$. If $P_1,\dots, P_k$ are the \srp s that use at least one instruction from $I$, and $Z_i$ is the event that
the edges in $P_i$ all survive, but every \srt~not using an instruction from $I$ has a failed edge, then $d_{I} = \Prob(Z_1\cup\dots\cup Z_k)$.
\end{lemma}

The following example will give an indication on how the minimum discrepancy of a graph can be determined, and we will use this result in the proof of Theorem~\ref{G crossing}. The discrepancies of all 10 forbidden minors for $A^*$ to be finite are computed in Chapter 4 of~\cite{Pat} for the case when $p$ is constant on $E(G)$.

\begin{example}\label{B0,p}
Let $B_0$ be the graph from Example~\ref{B0} and let $p,p_1,p_2\in(0,1)$. Let $p'$ be the probability distribution on $E(B_0)$ given by assigning probability $p$ to every edge, except that $s1$ receives probability $p_1$ and $s2$ receives probability $p_2$. Let $P=s,1,4,3,2,5,r$ and $P'=s,2,5,3,1,4,r$. Let $X$ be the event that the 6 edges in $P$ survive and all others fail. Define $X'$ for $P'$ similarly.

To find the minimum discrepancy $\dhat$ of $(B_0,p')$, consider the instruction 432. Observe that $A^*-432$ is finite, since this instruction is used in the only essential circuit $C$. The only \srp~using 432 is $P_1=P$. With the notation from Lemma~\ref{def lemma} we get that $Z_1=X$ since every $P+e$ contains an \srp~not containing 432. Thus by Lemma~\ref{def lemma}, $d_{432}=\Prob(X)=p^5p_1(1-p)^3(1-p_2)$. A similar argument with $P'$ shows that $d_{531}=\Prob(X')=p^5p_2(1-p)^3(1-p_1)$. 

Thus $\dhat\le \min\{d_{432},d_{531}\}=p^5(1-p)^3(\min\{p_1,p_2\}-p_1p_2)=d'$. To see that equality holds, consider $d_I$ for other sets $I$ such that $A^*-I$ is finite. Observe first that if $I$ contains no instruction contained in $P$ or $P'$, then every instruction contained in $C$ is in $A^*-I$ and is essential for $A^*-I$, so that $A^*-I$ is not finite.  Moreover, if $I$ contains one of $432$ or $531$, then it follows by monotonicity that $d_I\ge d'$ as desired.

Suppose now that $I$ contains an instruction $uvw$ contained in $P$ other than 432. We again have $P_1=P$, but in every case there is also a different \srp~$P_2$ containing $uvw$. Thus $d_I\ge\Prob(Z_1\cap Z_2)>\Prob(X)=d_{432}\ge\dhat$. A similar argument shows that if $I$ has an instruction contained in $P'$ other than 531, then $d_I>\Prob(X')\ge \dhat$.

It follows that $\dhat=d'=p^5(1-p)^3(\min\{p_1,p_2\}-p_1p_2)$.  Specifically, if $p=p_1=p_2$, then $\dhat=p^6(1-p)^4$.   
\end{example}

Combining this example with Proposition~\ref{prop:replace} and the observation that neither of $s1,s2$ is in an essential circuit we obtain the following proposition which we will use in Section~\ref{sec:piecewise}.

\begin{proposition}\label{B0+}
Let $H_1,H_2$ be series-parallel graphs and $G$ be the graph obtained from $B_0-\{s1,s2\}$ by identifying $s$ in $H_1,H_2$ with $s$, and identifying $r$ in $H_1,H_2$ with 1 and 2 respectively as shown in Figure~\ref{fig:B0 replace s1,s2}. If $p\in(0,1)$ is fixed and $p_i=\rho(H_i,p)$, then $\dhat(G,p)=p^5(1-p)^3(\min\{p_1,p_2\}-p_1p_2)$.
\end{proposition}

\begin{proof}
$\dhat(G,p)=\rho(G,p)-\rh(G,p)=\rho(B_0,p')-\rh(B_0,p')=\dhat(B_0,p')=p^5(1-p)^3(\min\{p_1,p_2\}-p_1p_2)$.
\end{proof}

\section{Crossings of protocol reliability functions}\label{sec:crossing}

It is natural ask if for given graphs $G$, $H$ it must be the case that $\rh(G, p)< \rh (H,p)$ for all $p\in (0,1)$ or $\rh(H,p)< \rh(G,p)$ for all $p\in(0,1)$. We will adapt an idea of Kelmans~\cite{Kel} to show that this need not be the case in a very strong sense. Following his approach we let 
$\drho(G,H)=\rho(G,p)-\rho(H,p)$, ${\drh}(G,H)=\rh(G,p)-\rh(H,p)$ and we say that the {\em profile} of a function is $(m_1,m_2,\dots,m_k)$ if it has exactly $t$ zeroes $x_1,\dots,x_t$ in $(0,1)$ with $0<x_1<x_2<\dots<x_t<1$ and $x_i$ has multiplicity $m_i$. In this language our original question is whether the profile of $\drh(G,H)$ must be empty. The following lemma now gives a simple negative answer for our question. 

\begin{lemma}\label{path-cycle} 
If $G=P_{k}$ and $H=P_{k+1}||P_{k+1}$ for $k\ge 2$, then $\drh(G,H)$ has profile (1) with zero $x_1=\gamma_k$ the unique root of $f_k(p)=p^k+p^{k-1}+\dots+p^2+p-1$ in (0,1).
\end{lemma}

\begin{proof} Since $G$ and $H$ are series-parallel, it follows that $$\drh(G,H)=\drho(G,H)=\rho(P_{k})-\rho(P_{k+1}||P_{k+1})=p^{k-1}-[2p^k-(p^k)^2]=p^{k-1}(p-1)f_k(p).$$ 
The profile of this function is the profile of $f_k$, which has a unique root $\gamma_k$ of multiplicity 1 in (0,1), since $f_k(0)=-1, f_k(1)=k-1>0$ and $f_k'$ is positive on $(0,1)$.
\end{proof}

The main idea to show that our original question has a negative answer in a much stronger sense is 
\begin{theorem}[Kelmans~\cite{Kel}, 4.1] \label{kelmans}
If $H_1=(F_1\circ G_1)||(G_2\circ F_2)$ and $H_2=(F_1\circ G_2)||(G_1\circ F_2)$, then $\drho(H_1,H_2)=\drho(F_1,F_2)\cdot \drho(G_1,G_2)$.
\end{theorem} 

Observe that if $F_1,F_2,G_1,G_2$ are all series-parallel in this statement, then so are $H_1$ and $H_2$, and we also get $\drh(H_1,H_2)=\drh(F_1,F_2)\cdot \drh(G_1,G_2)$. We can view this pair $(H_1,H_2)$ as a natural composition $(F_1,F_2)*(G_1,G_2)$ of the pairs $(F_1,F_2)$ and $(G_1,G_2)$. Thus if we compose the pair $(P_k, P_{k+1}||P_{k+1})$ from Lemma~\ref{path-cycle} with itself $m_k$ times we get a pair of graphs $(G_k,H_k)$ with profile ($m_k$) for $\drho=\drh$, where the unique root $\gamma_k$ has multiplicity $m_k$. Since $0=f_k(\gamma_k)<f_{k+1}(\gamma_k)$ for all $k\ge 2$ it follows moreover that $\gamma_2>\gamma_3>\dots$. So if we take such pairs $(G_k, H_k)$ for $2\le k\le t+1$ and compose them with each other we get a pair $(G,H)$ with profile $(m_{t+1},m_t,\dots,m_3,m_2)$ and $m_k$ is the multiplicity of the root $\gamma_k$. Relabeling the subscripts now we have proved the following.

\begin{theorem}\label{GH crossing}
For every $t$-tuple of positive integers $(m_1,m_2,\dots,m_t)$ there are series-parallel graphs $H_1, H_2$ such that $\drho(H_1,H_2)=\drh(H_1,H_2)$ has profile $(m_1,m_2,\dots m_t)$.
\end{theorem} 

\section{Piecewise polynomial optimal reliability functions}\label{sec:piecewise}

As we observed in Section~\ref{sec:prob} $\rh(G,p)$ can be achieved by different protocols for different $p$, so that $\rh$ may be piecewise polynomial.
For us a {\em breakpoint of order $m$} in a piecewise polynomial function $f$ will be a $z\in(0,1)$ such that $f$ is differentiable $m-1$ times at $z$, but not $m$ times, where 0 times differentiable means continuous at $z$. Equivalently there are different polynomials $p_1,p_2$ and $\varepsilon>0$ such that $f(x)=p_1(x)$ for $x\in(z-\varepsilon,z]$ and $f(x)=p_2(x)$ for $x\in[z,z+\varepsilon)$, and $p_1(x)-p_2(x)$ has a zero of order $m$ at $x=z$. Observe that $\min\{p_1(x),p_2(x)\}$ has a breakpoint (of order $m$) at $z$ if and only if $p_1(x)-p_2(x)$ has a zero of {\em odd} order $m$ at $x=c$. 

\begin{theorem}\label{G crossing}
For every $t$-tuple of positive odd integers $(m_1,m_2,\dots,m_t)$  there is a graph $G$ so that $\rh(G,p)$ is a piecewise polynomial with exactly $t$ breakpoints $x_1<x_2<\dots <x_t$ in (0,1) such that $x_i$ has order $m_i$.
\end{theorem}

\begin{proof}
Let $(H_1,H_2)$ be the series parallel graphs with profile $(m_1,m_2,\dots,m_t)$ for $\drho$ obtained from Theorem~\ref{GH crossing}.
Expanding $B_0$ at $s1$, $s2$ by $H_1$, $H_2$ as in Proposition~\ref{B0+} we obtain a graph $G$ with $\dhat(G,p)=p^5(1-p)^3(\min\{p_1,p_2\}-p_1p_2)$ where $p_i=\rho(H_i,p)$. This function has a breakpoint of order $m$ if and only if $p_1-p_2$ has a zero of multiplicity $m$, where $m$ is odd.
\end{proof}

\begin{remark}
Every breakpoint is an algebraic number, so not every number in (0,1) is a breakpoint of some $\rh(G,p)$, but by using suitable graphs of the form $P_{i_1}||P_{i_2}||\cdots|| P_{i_k}$ for $H_1, H_2$ in Theorem~\ref{G crossing} it should be possible to prove that the set of breakpoints is dense in (0,1). 
\end{remark}

The next theorem gives infinitely many small intervals that do not contain breakpoints for $\rh(G,p)$. 

\begin{theorem}\label{break points}
 For any rational $\frac ab\in[0,1]$,  $\rh(G,p)$ has no breakpoint $p$ with $0<|p-\frac ab|< (3b)^{-|E(G)|}$.

\end{theorem}

\begin{proof}
Let $m=|E(G)|$ and $c=\frac ab$.
It is sufficient to show that if $A,A'$ are finite protocols with $\rho_A(G,p)\neq\rho_{A'}(G,p)$, then the polynomial $f(x)=\rho_A(G,x+c)-\rho_{A'}(G,x+c)$ has no zero $x$ with $|x|<(3b)^{-m}$. Observe that for every protocol $A$ we can write $\rho_A(G,p)=\sum_{i=0}^ma_ip^i(1-p)^{m-i}$ where $a_i$ is the number of sets on $i$ edges that contain an $A$-walk. Thus $0\le a_i\le\binom mi$ and it follows that for some $b_i$ with $|b_i|\leq \binom mi$
\begin{align*}f(x)&=\sum_{i=0}^mb_i(x+c)^i(1-c-x)^{m-i}=\sum_{i=0}^m\sum_{j=0}^{i}\sum_{\ell=0}^{m-i}b_i\binom {i}jx^jc^{i-j}\binom{m-i}\ell(-x)^\ell(1-c)^{m-i-\ell}\\
&=\sum_{r=0}^m\sum_{j=0}^r\sum_{i=j}^{m-r+j}x^rb_i\binom {i}jc^{i-j}\binom{m-i}{r-j}(-1)^{r-j}(1-c)^{m-r-(i-j)}\\
&=\sum_{r=0}^mx^r\sum_{j=0}^r\sum_{s=0}^{m-r}(-1)^{r-j}b_{s+j}\binom {s+j}j\binom{m-s-j}{r-j}c^{s}(1-c)^{m-r-s}=\sum_{r=0}^mc_rx^r,
\end{align*}
where we used the substitutions $r=\ell+j$ and $s=i-j$. Observe that $1-c=\frac{b-a}b$ so that $c_r$ is rational with denominator $b^{m-r}\le b^m$ and we have that $|c_r|\ge b^{-m}$ unless $c_r=0$. Thus
\begin{align*}|c_r|&\le\sum_{j=0}^r\sum_{s=0}^{m-r}|b_{s+j}|\binom {s+j}j\binom{m-s-j}{r-j}c^{s}(1-c)^{m-r-s}\\
&\le\sum_{j=0}^r\sum_{s=0}^{m-r}\binom m{s+j}\binom {s+j}j \binom{m-s-j}{r-j}c^{s}(1-c)^{m-r-s}\\
&=\sum_{j=0}^r\sum_{s=0}^{m-r}\binom mr \binom {r}j\binom{m-r}{s}c^{s}(1-c)^{m-r-s}\\
&=\sum_{j=0}^r\binom mr \binom {r}j(c+1-c)^{m-r}=\binom mr2^r<\sum_{r=0}^m\binom mr 2^r1^{m-r}=3^m.
\end{align*}

If the degree of $f$ is $\alpha$, then $g(x)=x^\alpha f(\frac 1x)$ is a polynomial with leading coefficient $c_\beta\neq 0$, so by Cauchy's bound~(\cite{Cau}~p122) every zero $z$ of $g$ satisfies $|z|\le 1+\max\{|c_r/c_\beta|: 0\le r\le m\}\le 3^m/(1/b^m)=(3b)^m$ and thus the nonzero roots of $f$ are bounded below by $(3b)^{-m}$. 
\end{proof}

\section{Reliable protocols for probabilities near zero}
Determining $\rh(G,p)$ for all values of $p$ appears to be a difficult problem, however there is something we can say for $p$ close to zero.
There must be a protocol $A_0$ and an $\varepsilon>0$ such that $\rh(G,p)=\rho_{A_0}(G,p)$ for all $p\in[0,\varepsilon)$, and we call this the optimal protocol near zero. Near zero a good protocol $A$ will have many short $A$-paths. Consider 
$A^{(m)}=\{uvw: uvw$ is contained in some \srp~of length $\le m\}$ and let $d(u,v)$ denote the distance between $u$ and $v$ in $G$.

\begin{theorem}\label{near zero}
If $k=d(s,r)$ and $d_i$ is the number of \srp s of length $i$, then $A_0\supseteq A^{(k+1)}$ and $\rh(G,p)=\rho_{A_0}(G,p)=d_kp^k+ d_{k+1}p^{k+1}+O(p^{k+2})$ for all $p<3^{-m}$, where $m=|E(G)|$.
\end{theorem} 

\begin{proof} Observe that if we let $q=1-p$, and $G$ has $m$ edges, then for every protocol $A$ we can let $\rho_A(G,p)=\sum_{i=0}^ma_ip^iq^{m-i}$ where $a_i$ is the number of sets on $i$ edges that contains an $A$-walk. Observe that $a_i=0$ for all $i<k$ and $a_k\le d_k$, since every $A$-walk contains an \srp. The optimal protocol near zero must have $a_k$ as large as possible (as $p^kq^{m-k}$ is the dominant term), and subject to that it must have $a_{k+1}$ as large as possible (and so on.) Since every $A$-walk on a set of at most $k+1$ edges must contain an \srp~of length at most $k+1$ we conclude that $A_0\supseteq A^{(k+1)}$ if the latter is a finite protocol. If this the case, then we see that for $A_0$ we have $a_k=d_k$ (as every \srp~of length $k$ is now an $A_0$-walk) and $a_{k+1}=d_{k+1}+(m-k)d_k$ since every set of $k+1$ edges containing an \srw~must either be an \srp~of length $k+1$ or an \srp~of length $k$ and one additional edge. Thus
\begin{align*}
\rho_{A_0}(G,p)
&=d_kp^k(1-p)^{m-k}+(d_{k+1}+(m-k)d_k)p^{k+1}(1-p)^{m-k-1}+O(p^{k+2})\\
&=d_kp^k+d_{k+1}p^{k+1}+O(p^{k+2}).\\
\end{align*}
The bound of $3^{-m}$ follows from Theorem~\ref{break points} with $\frac ab=\frac 01$. It remains to see that $A^{(k+1)}$ is finite. This will follow if we can show that for all $uvw\in A^{(k+1)}$ we have $d(s,u)<d(s,w)$ since the distance from the vertices in an essential circuit to $s$ would have to increase for every two steps along the circuit, but eventually we will repeat a vertex as we continue along the circuit. So suppose that $d(s,u)\ge d(s,w)$ and let $Q$ be a path of length at most $k+1$ that contains the instruction $uvw$. The length of the $s,w$ segment of $Q$ is at least $d(s,u)+2$ and so if we replace this segment by a shortest $s,w$-path, then we get an \srw~of length at most $(k+1)-(d(s,u)+2)+d(s,w)\le k-1$, a contradiction.
\end{proof}

\section{Reliable protocols for probabilities near one}
As we observed previously, for every protocol $A$ we can let $\rho_A(G,p)=\sum_{i=0}^ma_ip^iq^{m-i}$ where $a_i$ is the number of sets on $i$ edges that contains an $A$-walk. The optimal protocol near one must have $a_m$ as large as possible (as $p^mq^{0}$ is the dominant term), and subject to that it must have $a_{m-1}$ as large as possible (and so on.) If $c_k$ is the number of edge-sets of size $k$ that is contained in some edge-cut that disconnects $s$ from $r$, then clearly $a_i\le \binom{m}{i}-c_{m-i}$. It is the main result of~\cite{KPR} that if the size of a smallest edge-cut separating $s$ and $r$ is $e$, then there is a finite protocol $A'$ such that there is an $A'$-walk in $G$ unless at least $e+1$ edges fail or the $e$ edges of a minimum cut separating $s$ and $r$ fail. So if we let $A_1$ be the optimal finite protocol near one then for $A_1$ we get $a_i=\binom mi$ when $i>m-e$, and $a_{m-e}=\binom m{m-e} - c_{e}$, where $c_e$ simply counts the number of edge-cuts of minimum size $e$ that separate $s$ from $r$. It now follows that 
\begin{align*}
\rho_{A_1}(G,p)
&=\sum_{i=0}^ma_ip^iq^{m-i}=\sum_{i=m-e}^m\binom mi p^iq^{m-i} -c_{e}(1-q)^{m-e}q^e+O(q^{e+1})\\
&=\sum_{i=0}^m\binom mi p^iq^{m-i} -c_{e}q^e+O(q^{e+1})=(p+q)^m -c_{e}q^e+O(q^{e+1})=1 -c_{e}q^e+O(q^{e+1}).\\
\end{align*}
This proves the following counterpart to Theorem~\ref{near zero} for the optimal protocol for probabilities near one.
\begin{theorem}\label{near one}
$\rh(G,p)=\rho_{A_1}(G,p)=1-c_eq^e+O(q^{e+1})$ for $q<3^{-m}$.
\end{theorem} 

To improve on this result in general would require a better understanding of the notion of robustness studied in~\cite{KPR}. 
A protocol $A$ is called {\em $k$-robust} if it is finite, and for every set $E$ of at most $k$-edges that does not disconnect $s$ from $r$, there is an $A$-walk in $G-E$. In a $k$-robust protocol we have $a_{m-i}=\binom mi -c_{i}$ for all $i\le k$, and the optimum protocol $A_1$ near one must have maximum robustness. Thus studying the properties of $A_1$ can be viewed as a refinement of the approach in~\cite{KPR}.

\section{Open problems}

Computing $\rh(G,p)$ exactly is likely to be a very hard problem in general, since Provan and Ball~\cite{PB} showed that even computing $\rho(G,p)$ is \#P-hard. 

The most interesting open question is clearly to characterize the graphs for which $\rh(G,p)$ is a polynomial. In~\cite{Pat} it is shown that the 10 minor-minimal graphs $G$ for which $\rh(G,p)<\rho(G,p)$ all have the property that $\rh(G,p)$ is polynomial, so such a characterization could be quite difficult to obtain. One point of inquiry could be to find all graphs for which $\rh(G,p)=\rho_{A^*-I}(G,p)$ where $I$ is a single instruction.

Can we give a polynomial time procedure for determining $\rh(G,p)$ exactly for a fixed $p$, or near 0 or near 1?  

\section*{Acknowledgements}

The first author thanks Johan Rosenkilde for a helpful discussion that led to Theorem~\ref{break points}.

\bibliographystyle{siam}
\bibliography{ProtocolReliability-arxiv}
\end{document}